\documentclass{amsart}
			\usepackage{amssymb, latexsym}
			\usepackage{url}
			
			\usepackage{amsmath}
			\usepackage{amsfonts}
			\usepackage{amsthm}
			\usepackage{textcomp}
			\usepackage{graphicx}
			\usepackage{mathrsfs}
			\usepackage{color}
			\usepackage{units}
			\usepackage{enumerate}
			\usepackage{esint}
			\usepackage{caption}
			\usepackage{subcaption}
			
			\usepackage{float}
			
			\usepackage{epstopdf}
			\usepackage[bookmarks, pdftitle={BD Dichotomy}, pdfauthor={Michael Kelly and Lorenzo Sadun}, colorlinks=true, linkcolor=black, citecolor=black, urlcolor=black]{hyperref}

\begin{document}
			 \bibliographystyle{plain}
			 
			 
			 \newtheorem{theorem}{Theorem}
			 \newtheorem{lemma}[theorem]{Lemma}
			  \newtheorem{proposition}[theorem]{Proposition}
			 \newtheorem{corollary}[theorem]{Corollary}
			 \newtheorem{conjecture}[theorem]{Conjecture}
			 \newtheorem*{definition}{Definition}
			  \newtheorem*{problem}{Problem}
			   \newtheorem{construction}{Construction}
			  \newtheorem{remark}{Remark}
	  \newtheorem{example}[theorem]{Example}
			
			  
		 \newcommand{\R}{\mathbb{R}}   
		 \newcommand{\T}{\mathbb{T}}
		 \newcommand{\Q}{\mathbb{Q}}
		 \newcommand{\Z}{\mathbb{Z}}
		\newcommand{\lb}{\left\lbrace}
		\newcommand{\rb}{\right\rbrace}
		\newcommand{\lp}{\left(}
		\newcommand{\rp}{\right)}
		\newcommand{\ar}{\right\rangle}
		\newcommand{\al}{\left\langle}
		\newcommand{\ra}{\rightarrow}
		\newcommand{\dint}{\displaystyle\int}
				\newcommand{\dsum}{\displaystyle\sum}
		\newcommand{\by}{\boldsymbol{y}}
		 \newcommand{\ds}{\displaystyle}			      
		\newcommand{\disc}{\mathrm{disc}} 
		 \newcommand{\length}{\mathrm{Length}}


\title[PE Cohomology \& Kesten's Theorem ]{Pattern Equivariant Cohomology \\ and Theorems of Kesten and Oren}
\author[Kelly ]{Mike~ Kelly}
\author[Sadun]{Lorenzo~Sadun}
\date{\today}


\address{Department of Mathematics, Univerisity of Texas, Austin, Texas 78712 USA}
\email{mkelly@math.utexas.edu}
\email{sadun@math.utexas.edu}

\numberwithin{equation}{section}


\maketitle

\begin{abstract} 
In 1966 Harry Kesten settled the Erd\H os-Sz\"usz conjecture on the local discrepancy of irrational rotations. His proof made heavy use of continued fractions and Diophantine analysis. In this paper we give a purely 
topological proof Kesten's theorem (and Oren's generalization of it) using the pattern equivariant cohomology of aperiodic tiling spaces. 
\end{abstract}

\section{Introduction}
Let $\T=\R/\Z$ and for an irrational number $\xi\in\R$ define a map $T_{\xi}:\T\ra\T$ by
	\[
		T_{\xi}(x)=x+\xi\mod\Z
	\]
It is well known that 
for any interval $I\subset \T$, the function
	\[
		D(N)=D(N;I)=\# I\cap \lb T_{\xi}^{k}(x) : 0\leq k\leq N  \rb -N\length(I)
	\]
is $o(N)$ for each $x\in\T$. That is, $D(N)/N\ra 0$ as $N\ra \infty$ for each $x$. $D(N)$ is sometimes called the  {\it local discrepancy} or {\it error} function. As early as the 1920's, it was known (by Hecke and Ostrowski  \cite{Hecke,Ostrowski}) that if there exists an integer $k$ such that
	\begin{equation}\label{KestenCondition}
		\length(I) \equiv k\xi\mod \Z
	\end{equation}
 then $D(N)=O(1)$. That is, $D(N)$ is {\it bounded}. In the 1960's it was conjectured by  Erd\H os-Sz\"usz  \cite{Erdos} -- and proved by Kesten \cite{Kesten} -- that the converse holds. 
 
\begin{theorem}[Kesten] Let $I\subset \T$ be an interval, and $\xi\in\R$ be irrational. There exists a constant $C>0$ such that $|D(N)|<C$ if, and only if, (\ref{KestenCondition}) holds. 
\end{theorem}
A generalization of Kesten's theorem for several disjoint intervals is given in the following theorem of Oren \cite{Oren}.
\begin{theorem}[Oren] \label{Oren} Let $I_{1},...,I_{L}\subset \T$ be $L$ disjoint intervals and $\xi\in\R$ be irrational. There exists a constant $C>0$ such that $|D(N;I_{1})+\cdots D(N;I_{L})|<C$ if, and only if, there is a permutation $\sigma$ such that $b_{\sigma(\ell)}-a_{\ell}\equiv k_{\ell}\xi \mod \Z$ for some $k_{\ell}\in \Z$. Here $I_{\ell}=[ a_{\ell},b_{\ell}  ]$. 
\end{theorem}
In this paper we give purely  topological proofs of Kesten's and Oren's Theorems using the pattern equivariant cohomology of aperiodic tiling spaces. We reformulate and prove these theorems in the context of ``cut-and-project patterns'' and the ``Bounded Displacement (BD) equivalence relation.''  Our proof is based upon a recent topological rigidity result for {\it model sets} \cite{KS14}. \newline

\indent  If $S\subset \R^2$ is a strip\footnote{For readers who are familiar with cut-and-project tilings, the strip $S$ is simply $S=V+W$ where $V$ is the acting subspace and $W$ is the window.}, then we will use the notation
\[S(\Z)=S\cap \Z^2\]
to denote the $\Z$-points of $S$. If $S$ has irrational slope, then the projection of $S(\Z)$ onto a line
parallel to $S$ is called a {\it cut-and-project set} or a {\it model
  set.}\footnote{Cut-and-project sets can also be obtained from
  lattices in more than 2 dimensions, but in this paper we only
  consider those arising from 2-dimensional strips.}  To see the connection between Kesten's result and tilings,\footnote{This connection seems to be well known. See for instance \cite{scoop,H,HK,KSV}.} notice that there is a one-to-one correspondence (given by the projection onto the $x$-axis) between  the $\Z$-points of the strip\footnote{Here $\tilde{x}$ and $I$ are identified with their coset representatives in $[0,1)$.}
	\[
		S_{\xi,I}=\lb (x,y) : \xi x-y-\tilde{x}\in I \rb
	\]	 
and integers $k$ such that $T_{\xi}^{k}(\tilde{x})\in I$. So the discrepancies of the sequence $T_{\xi}(x),T_{\xi}^{2}(x),...$ and the associated cut-and-project set are one and the same.\newline

A current object of interest in tiling theory is the {\it BD equivalence relation}.\footnote{See \cite{APCG,H,HKW,HK,Solomon1,Solomon2} for recent developments and \cite{scoop,Laczkovich92,ST} for some earlier developments. BD equivalence is sometimes referred to as {\it wobbling equivalence} \cite{BG,DSS}.} Two subsets $Y_{1},Y_{2}\subset \R^N$ are said to be BD if there is a bijection $\varphi:Y_{1}\ra Y_{2}$ such that 
	\[
		\ds\sup_{\by\in Y_{1}} \| \by -\varphi(\by) \|<\infty.
	\]

It is not hard to see that a subset of $\R$ is BD to a lattice if and 
only if its discrepancy (in the sense of \S \ref{section:disc}) is bounded. (See \cite{Laczkovich92} for analogous statements in 
higher dimensions). Kesten's Theorem 
can then be restated in the language of aperiodic point
patterns and the BD equivalence relation as follows:

\begin{theorem}\label{mainTheorem}
  Let $Z$ be a 1 dimensional cut-and-project set obtained from an
  irrational strip in $\R^2$. $Z$ is BD to a lattice if and only if the
  boundary components of the strip are equivalent mod $\Z^2$.
\end{theorem}
%

%
%

\section{Preliminaries}
In this section we will review some of the concepts and results that
we will use in the proof of our main results. We will not present this
material in generality. The interested reader is encouraged to consult
the references for a detailed treatment of the ideas below.
%
%

\subsection{The topology of cut-and-project patterns}

A 2-to-1 cut-and-project pattern $Z$ is a subset of $\R^2$ obtained
from the following construction. Let $V$ and $H$ be transverse lines
in $\R^2$, and $W\subset H$ (the {\it window}) be a compact set that
is the closure of its interior. Let $\pi_{V}:\R^{2}\ra \R^{2}$ be a
linear projection of $\R^2$ onto $V$. Then
	\[
		Z=\pi_{V}\big(   (V+W)\cap \Z^{2}  \big),
	\]
        where $V+W=\lb {\bf v+w} ~:~ {\bf v}\in V,~{\bf w}\in W
        \rb$. If $\partial W$ has Hausdorff measure zero, then $Z$ is
        called {\it regular.} In this paper, $W$ is either an interval
        or a finite union of intervals, so $Z$ is always regular. \newline

Our main tool is the {\it pattern equivariant} cohomology of $Z$
(\cite{K1}, or see \cite{S1} for a review). We
think of $Z$ as the vertices of a tiling of $V\cong \R$ by
intervals. We abuse notation by denoting this tiling as $Z$.  A
function $f:V\ra \R$ is said to be {\it strongly} pattern equivariant,
or strongly PE, if there exists an $R>0$ such that for any ${\bf
  v},{\bf v}^{\prime}\in V$ such that $B_{R}({\bf v})\cap Z$ and
$B_{R}({\bf v}^{\prime})\cap Z$ are translates of each other, we have
$f({\bf v})=f({\bf v}^{\prime})$. A function is {\it weakly} PE if it
is the uniform limit of strongly PE functions. We can similarly speak
of strongly and weakly PE 0-cochains that are evaluated on the
vertices of $Z$ and 1-cochains that are evaluated on the edges of $Z$
(and so on for higher-dimensional tilings. In our case the cochain
complex ends at dimension 1).  The coboundary $d\alpha$ of a (strongly
or weakly) PE cochain $\alpha$ is easily seen to be (strongly or
weakly) PE.\footnote{We denote the coboundary by $d$ since $\delta$
  denotes the density of a point pattern.}

The cohomologies of the resulting cochain complexes are called the
(strong or weak) PE cohomologies of $Z$, and are denoted $H^*_{s}(Z)$
and $H^*_{w}(Z)$.  
Kellendonk \cite{K1} (in a slightly different
setting) and Sadun \cite{S2} (in this setting) showed that the strong
PE cohomology of $Z$ is isomorphic to the \v Cech cohomology of the
associated tiling space, and that the strong PE cohomology with real
coefficients is isomorphic to the \v Cech cohomology with real
coefficients. The weak PE cohomology (necessarily with real or complex
coefficients) is much more complicated.

By definition, a nontrivial class in $H^1_s(Z)$ can never be
represented by the coboundary of a strongly PE 0-cochain. However, it
sometimes can be represented by the coboundary of a weakly PE
0-cochain. If so, the class is called {\it asymptotically negligible}
\cite{CSshape, K2}, in which case {\it every} representative of the class
is of this form. Let $H^1_{an}(Z) \subset H^1_s(Z)$ denote the
asymptotically negligible classes. These classes are described by the
following lemma, which is essentially Corollary 4.4 from \cite{KS13}.
\begin{lemma}\label{weakbdLemma}
  For a closed strongly PE 1-cochain $\alpha$, there is a 0-cochain
  $\beta$ such that $\alpha= d\beta$. 
Furthermore, $\beta$ is weakly
  pattern equivariant if, and only if, $\beta$ is bounded.
\end{lemma}

We now use a recent result about $H^1_{an}$ for cut-and-project sets.
The following theorem is a special case of a theorem from \cite{KS14}.
\begin{theorem}\label{KS14Theorem}
  Let $Z$ be a 2-to-1 dimensional cut-and-project set whose window $W$
  is an interval or a finite union of intervals. Then
  $H^{1}_{an}(Z)$ is one dimensional and is generated by the
  differential of the coordinate function on $H$.
\end{theorem} 
%
%

\subsection{Discrepancies and the BD equivalence relation}\label{section:disc}
Given a discrete subset $Y$ of $\R$, a number $\delta>0$, and an
interval $I$, we define the {\it discrepancy} of $Y$ with respect to
$\delta$ and $I$ to be
	\[
        \disc_{Y}(I,\delta)=\left| \# I\cap Y -\delta\length(I).
        \right| \] 

If there exists a $\delta>0$ for which
        $\disc_{Y}(I,\delta)=o(\length(I))$, then one expects $\#I\cap
        Y\approx \delta\length(I)$ for large intervals $I$. If such a
        number $\delta>0$ exists, then it is unique and it is called
        the {\it density} of $Y$. Hence with the correct choice of
        $\delta>0$ (if it does exist), the discrepancy is a measure of
        error of the expected number of points of $Y$ in $I$, versus
        the true number of points.

The following is a special case of a theorem of Laczkovich 
\cite{Laczkovich92}, and can also be easily proved directly: 

\begin{theorem}\label{LaczkovichTheorem}
  For a discrete subset $Y$ of $\R$ and $\delta>0$, the following are
  equivalent:
	\begin{enumerate}[(i)]
		\item $Y$ is  BD to a lattice of covolume $\delta^{-1}$. 
		\item There exists a constant $c>0$ such that for
                  every finite interval $I$
			\[
				\disc_{Y}(I,\delta)<c.
			\]
	\end{enumerate}
\end{theorem}

We can similarly define the discrepancy of any pattern, or of any
strongly PE 1-cochain. Every such 1-cochain $\alpha$ can be written as
a finite linear combination
\[ \alpha = \sum_j c_j \chi(P_j), \] 
where the {\it indicator cochain} $\chi(P_j)$ evaluates to 1 on a
particular edge of the pattern $P_j$ and to zero on all other edges.
Let
\[ \alpha_0 = \sum_j c_j [\chi(P_j) - \delta(P_j) dx], \] 
where $\delta(P_j)$ is the density of $P_j$ and $dx$ is the 
1-cochain that assigns to each
edge its length. (These densities are well-defined thanks to the
unique ergodicity of cut-and-project sets.) The discrepancy of
$\alpha$ over an interval is $\alpha_0$ applied to that interval. This
is a linear combination of the discrepancies of the patterns $P_j$.

A cochain $\alpha$ has bounded discrepancy if and only if $\alpha_0$
has bounded integral, which is if and only if $\alpha_0$ represents an
asymptotically negligible class. Equivalently, $\alpha$ has bounded
discrepancy if and only if the cohomology class of $\alpha$ is a
linear combination of a class in $H^1_{an}$ and the class of $dx$.
The following is then an immediate corollary of Theorem
\ref{KS14Theorem}:

\begin{corollary}\label{2DimCor}Let $Z$ be a 2-to-1 dimensional
  cut-and-project set whose window $W$ is an interval or a finite
  union of intervals. Then the set of 1-cohomology classes that are
  represented by cochains with bounded discrepancy is two dimensional.
\end{corollary}

%
%
\section{Proof of Theorem \ref{mainTheorem}}
Let $S$ be an irrational strip, and write $S=V+W$ where $V$ is a one dimensional
subspace of $\R^2$ and $W\subset H$ (a subspace transverse to $V$) is
a closed interval. 

\begin{proof}

We assume without loss of generality that $\partial S \cap \Z^2$ is empty,
since there are at most two points in $\partial S \cap \Z^2$ and these 
do not affect whether the discrepancy is bounded. 

  The \v Cech cohomology of a tiling space $\mathcal{T}$ associated
  with a non-singular 2-to-1 dimensional cut-and-project set whose window is an
  interval is well understood \cite{FHK}. The space is homeomorphic to
  a ``cut torus'', obtained by taking $\T^2$, removing a copy of
  $\pi(\partial S)$, and gluing each point back in twice, once as a
  limit from one side and once as the limit from the other side.  The
  resulting space has the cohomology of a once- or twice-punctured
  torus, depending on whether $\pi(\partial S)$ consists of one or two
  path components.  In particular, if the boundaries $\ell_{1,2}$ are
  related by an element of $\Z^2$, then $H^1_s(Z) = \R^2$, while
if the boundaries are not related then $H^1_s(Z) = \R^3$,
  since $H^1$ of a once- or twice-punctured torus is 2- or 3-dimensional.

Suppose that the two components of $\partial S$ are equivalent 
(mod $\Z^2$), and hence that $H^1_s(Z) = \R^2$.  
  By Corollary \ref{2DimCor}, the cohomology classes of 1-cochains
with bounded discrepancy is also 2 dimensional, so {\it
  all} classes in $H^1$ are represented by cochains with bounded
discrepancy. Adding the coboundary of a strongly PE 0-cochain to a
1-cochain does not change the boundedness (or unboundedness) of the
discrepancy of that 1-cochain, so in fact all 1-cochains have bounded
discrepancy. This shows not only that the cut-and-project set $Z$ 
has bounded discrepancy (and so is
BD to a lattice), but also that {\it any point pattern $Z'$ locally
  derived from $Z$ is BD to a lattice}.

If the two components of $\partial S$ are not equivalent, then 
$H^1_s(Z)$ is strictly larger than the set of 1-cochains with bounded
discrepancy, so 
there exists a strongly PE cochain
  \[ \alpha = \sum_j c_j \chi(P_j) \] 
  with unbounded discrepancy. The discrepancy of $\alpha$ is a linear
  combination of the discrepancies of the indicator cochains
  $\chi(P_j)$, so at least one of the patterns $P_j$ must have
  unbounded discrepancy.  Let $P$ be such a pattern with unbounded
  discrepancy.

  The indicator cochain $\chi_P$ evaluates to 1 on edges whose left
  endpoints are projections of points in an ``acceptance domain'' $V+
  \tilde W$, where $\tilde W \subset W$.  $\tilde W$ is obtained by
  applying the condition that a certain finite set of points must
  appear in the pattern, and another finite set must not appear. As
  such, $\tilde W +V$ is the intersection of a finite number of
  translates of $S$ by fixed elements of $\Z^2$ and a finite number of
  translates of $\R^2 \backslash S$ by fixed elements of
  $\Z^2$. $\tilde W$ can thus be written as a disjoint union of
  finitely many intervals $W_i$, each of whose boundary components are
  related to the boundaries of $W$ by elements of $\Z^2$.

  Since the multi-strip $\tilde W + V$ has unbounded discrepancy, at
  least one of the strips $\tilde W_i + V$ must have unbounded
  discrepancy. Let $\ell'_{1,2}$ be the boundaries of $\tilde W_i
  +V$. By Theorem \ref{mainTheorem} with $\pi(\partial S)$ path
  connected, which we have already proven, $\ell'_1$ and $\ell'_2$
  cannot be equivalent (mod $\Z^2$). Thus $\ell_1'$ must be equivalent
  to one component $\ell_1$ of $\partial S$ and $\ell_2'$ must be
  equivalent to the other component $\ell_2$.

  We apply Theorem \ref{mainTheorem} with $\pi(\partial S)$
  path-connected yet again. The strip between $\ell_1$ and $\ell_1'$
  has bounded discrepancy, and the strip between $\ell_2$ and $\ell_2'$
  has bounded discrepancy, and the strip between $\ell_1'$ and
  $\ell_2'$ has unbounded discrepancy, so the strip between $\ell_1$
  and $\ell_2$ must have unbounded discrepancy. But that is precisely
  $Z$. Since $Z$ has unbounded discrepancy, it cannot be BD to a
  lattice.
\end{proof}
%

%
\section{Proof of Oren's Theorem}
Our proof of Oren's theorem follows the same lines as our proof of
Theorem \ref{mainTheorem}. The main idea is to identify generators for
the cohomology with unbounded discrepancy (appealing again to
Corollary \ref{2DimCor} and also to Theorem \ref{mainTheorem}) and
observe that the class of the combined intervals yield the trivial
class exactly when the hypotheses of Oren's theorem are satisfied.
\begin{proof}[Proof of Theorem \ref{Oren}]
  Let $S$ be a disjoint union of strips $S_{1},...,S_{L}$ and suppose
  there are $n$ distinct boundary components of $S$ modulo $\Z^2$. In
  the notation of Theorem \ref{Oren} the strip $S_{\ell}$ is given by
	\[
		S_{\ell}=\lb (t,y)~:~ \xi t-y-x \in I_{\ell} \rb.
	\]

        Let $E$ be the convex hull of $S$ and let $\mathcal{T}$ be the
        {\it colored} cut-and-project set obtained in the following
        way. Color a point $p\in E(\Z)$ ``$\ell$'' if $p$ belongs to
        $S_{\ell}$ and ``$\omega$'' otherwise. Let $\mathcal{T}$ be
        the projection of $E(\Z)$ onto any line parallel to $E$

        Keeping in mind that $\mathcal{T}$ is colored, we have
        $H^{1}(\mathcal{T})=\R^{n+1}$ since the associated tiling
        space of $\mathcal{T}$ is a cut-torus with $n$ cuts. Let
        $H^1_{ud}$ be the quotient of $H^1(\mathcal{T})$ by the 
subspace of classes with bounded discrepancy.  
By Corollary \ref{2DimCor} this quotient space is
        $(n-1)$-dimensional. We will now describe a set of generators
        for $H^{1}_{ud}$.

        Let $L_{1},...,L_{n}$ be boundary components that represent
        each of the $\Z^2$ classes in $\partial S$, and let $B_{j}$ be
        the convex hull of $L_{1}$ and $L_{j}$. We claim the classes
        (in $H^1_{ud}$) of the indicator cochains $i_{B_j}$ of the
        $B_{j}$'s form a basis for $H^{1}_{ud}$.

        To see that these classes span, recall that $H^1$ is spanned
        by indicator functions of patterns, and that the acceptance
        domain of each pattern is a multi-strip whose boundaries are
        translates (in the vertical direction) by $\Z^2=\Z + \Z\alpha$
        of the various $B_j$'s. This means that the acceptance domain
        can be written as an integer linear combination of the
        $i_{B_j}$'s, plus (or minus) the indicators of some intervals
        whose lengths are in $\Z + \Z \alpha$. Since indicator
        functions of intervals whose lengths are in $\Z + \Z\alpha$
        have bounded discrepancy, these do not affect the class in
        $H^1_{ud}$.

        Since the $n-1$ classes of the $i_{B_j}$'s span $H^1_{ud}$,
        and since $H^1_{ud}$ is $(n-1)$-dimensional, these classes are
        linearly independent. The only way for a multi-slab to give
        the zero class is for the boundaries to cancel perfectly mod
        $\Z^2=\Z + \alpha \Z$, which is precisely the hypothesis of
        Oren's theorem.
\end{proof}

\section{Concluding Remarks}
The virtue of pattern equivariant cohomology is that it is not just
abstract nonsense---you get to see the cohomology work. In the above
proofs, the PE cohomology actually allows you to see what you're
counting. This feature (along with some simple observations about the
topology of the punctured torus, and some basic linear algebra) yields
Kesten's and Oren's theorems without any Diophantine analysis, thereby
demonstrating both the power and the intuitive appeal of PE
cohomology.

There is a large literature consisting of generalizations and reproofs
of Kesten's theorem (see \cite{BT, Ferenczi, GL, KL,Rauzy,Schmidt} for
a small sample), including cohomology-type proofs \cite{Helson} and
\cite{Halasz} using dynamical cocycles on $\T$. As far as we know,
this is the first proof of Kesten's theorem that deals directly with
the associated tilings.

We remark on one generalization of Kesten's theorem, the notion of a
{\it bounded remainder set} (BRS). This concept has been studied by a
number of authors, such as \cite{Ferenczi, GL, HK,
  Liardet,Rauzy}. Windows that are BRS's yield examples of
cut-and-project sets that are BD to lattices, as has been recently
reported in \cite{HK}.  With projections to spaces of dimension higher
than one, however, the notion of a bounded remainder set is too
strong---one can have windows that are not BRS's but still generate
cut-and-project sets that are BD to lattices.

Consider the following reformulation of Kesten's
theorem:\footnote{This formulation, related ideas, and similar
  results---especially in identifying the role of fundamental
  domains---were reported in \cite{scoop}.}
 \begin{quote}
 	{\it
		For an irrational strip $\mathrm{S}$, $S(\Z)$ is BD to a lattice if, and only if, $S$ is the closure of a fundamental domain of a cyclic subgroup of $\Z^2$.
	}
 \end{quote}
For higher dimensional spaces $V$, we have the following:
 \begin{quote}
 	{\it Let $S$ be an irrational slab\footnote{We will say that $S\subset \R^N$ is a slab if it is the closed convex hull of two distinct parallel codimension one hyperplanes. It is irrational if its boundary descends to a dense subset of $\T^N$.} in $\R^N$. If $S$ is the closure of  a fundamental domain for a cyclic subgroup of $\Z^N$, then the set $S(\Z)$ is BD to a lattice. 
	}
 \end{quote}
The proof of the above statement follows without much difficulty from the following observation: if $L$ is a cyclic subgroup of $\Z^N$, then $S(\Z)$ can be (modulo some points on the boundary) identified with $\Z^N/L$. But $\Z^N/L$ {\it is a lattice} in the quotient group $\R^N/L$! To show that $S(\Z)$ is BD to a lattice in $\R^N$ (not just in the quoitent) we appeal to simple variant of Proposition 2.1 from \cite{HKW} (where the group $\R^N$ is replaced with $\R^N$/L). This argument can also be made for subgroups of $\Z^N$ of higher rank, yielding a non-trivial family of cut-and-project sets that are BD to lattices. However, unlike in the 2-dimensional situation, the converse to the above result is false in general by Theorem 1.2 of \cite{HKW}.  
\section*{Acknowledgments}
We thank Jos\'e Aliste-Prieto, Natalie Priebe Frank, Nir Lev, Alan Haynes, Johannes Kellendonk and Barak Weiss
 for helpful discussions and for comments on earlier drafts of this 
paper. This work is partially supported by NSF grant DMS-1101326.

\nocite{*}
\bibliographystyle{plain}
\bibliography{slabs}	
\end{document}